\documentclass[reqno,onecolumn]{amsart}
\usepackage{eurosym}
\usepackage{amsfonts,amscd,amsthm}
\usepackage{amssymb}
\usepackage{amsmath}
\usepackage{graphicx}
\usepackage{amscd,amsthm}
\usepackage{hyperref}

\newtheorem{theorem}{Theorem}
\theoremstyle{plain}

\newtheorem{corollary}{Corollary}

\newtheorem{definition}{Definition}

\newtheorem{lemma}{Lemma}

\newtheorem{remark}{Remark}

\numberwithin{equation}{section}

\begin{document}

\begin{center}
\pagestyle{myheadings}\thispagestyle{empty}%
\markboth{\bf M. Bouzeraib, A. Boussayoud, }
{\bf The new combinatorial identities of symmetric ...}

\huge\textbf{The new combinatorial identities of symmetric functions}
\end{center}

\vspace{0.3cm}
\begin{center}
\textbf{Meryem Bouzeraib$^{1}$, Ali Boussayoud$^{1}$, Salah Boulaaras$^{2}$   }\\

$^{1}$ LMAM Laboratory and Department of Mathematics, \\
Mohamed Seddik Ben Yahia University, Jijel, Algeria\\
$^{2}$ Department of Mathematics, College of Sciences, Qassim University, Buraydah 51452, Saudi Arabia

E-Mails: meryembouzeraib@gmail.com; aboussayoud@yahoo.fr;  s.boularas@qu.edu.sa;

\vspace{0.3cm}
\textbf{\large Abstract}
\end{center}

\begin{quotation}
\qquad In this paper, we provide some novel 
binomial convolution  related to symmetric functions, as well as convolution sums without the binomial symbol. Moreover we give some new convolution sums of Bernoulli, Euler, and Genocchi numbers and polynomials  with symmetric functions 
, by making use of the
elementary methods including exponential generating functions. From these
convolutions we deduce several new combinatorial identities  for bivariate polynomials, and 
 we establish several new identities related  Bernoulli, Euler and Genocchi numbers and polynomials with certain  bivariate polynomials such as bivariate Fibonacci, bivariate Lucas, bivariate balancing and bivariate balancing-Lucas  polynomials.
\end{quotation}

\section{\textbf{Introduction and preliminaries}}
Symmetric functions are fundamental in several areas of mathematics, such as combinatorics, algebra, and representation theory. The concept of symmetry is essential for analysis and plays a key role in uncovering new insights and relationships.
A symmetric function is a formal power series in the variables 
$\{x_{1}, x_{2}, ..., x_{n}\}$
  that does not change when the variables are rearranged in any order. This means the function remains unchanged under any permutation of its variables. 
For instance,  
$f\{x_{1}, x_{2}, ..., x_{n}\}$  is a symmetric function if for every permutation $\alpha$ of the set  $\{x_{1}, x_{2}, ..., x_{n}\}$, the following relationship is holds:
\begin{equation*}
\sigma f(x_{1},x_{2},...x_{n})=f(x_{\sigma(1)},x_{\sigma(2)},...x_{\sigma(n)})=f(x_{1},x_{2},...x_{n}).
\end{equation*}

We present  a brief overview of the fundamental symmetric functions and some of their properties. For a more detailed explanation and proofs, consult Macdonald's book \cite{Macdonald}. For a positive integer, 
$n$ the following symmetric functions are defined:

\textbf{Elementary Symmetric Function}
\begin{equation*}
e_{k}^{(n)}=e_{k}(x_{1}, x_{2},...,x_{n})=\sum_{i_{1}+i_{2}+...+i_{n}=k}x_{1}^{i_{1}}x_{2}^{i_{2}}...x_{n}^{i_{n}}, \quad \left(0 \le k \le n \right), \label{sq}
\end{equation*} 
with
$i_{n}\text{ ,}...\text{ ,}i_{2} \text{ ,} \text{ ,}i_{1}$
equal
$0$
or 
$1$.

\textbf{Complete Homogeneous Symmetric Function}
\begin{equation*}
h_{k}^{(n)}=h_{k}(x_{1}, x_{2},...,x_{n})=\sum_{i_{1}+i_{2}+...+i_{n}=k}x_{1}^{i_{1}}x_{2}^{i_{2}}...x_{n}^{i_{n}}, \text{  } \left(0 \le k \right), \label{s43q}
\end{equation*}
with 
$i_{n},...,i_{2} ,i_{1}$
greater than or equal to 
$0$.

\textbf{Power sum symmetric function}

\begin{equation*}
p_{k}^{(n)}=p_{k}(x_{1}, x_{2},...,x_{n})=\sum_{i=1}^{n}x_{i}^{k}
\end{equation*}
\begin{definition}
\cite{Abderrezzak} Let $A$ and $E$ be any two alphabets. We define $%
S_{n}(A-E)$ by the following form:
\begin{equation*}
\frac{\prod\limits_{e\in E}(1-ez)}{\prod\limits_{a\in A}(1-az)}
=\sum\limits_{n=0}^{\infty }S_{n}(A-E)z^{n},  \label{jk}
\end{equation*}
with the condition $S_{n}(A-E)=0\ $for\ $n<0.$
\end{definition}
\begin{definition}
\cite{nabiha} Let $n$ be a positive integer and $A=\left\{ x_{1},x_{2}\right\}
\ $ be a set of given variables. Then, the $n^{th}$ symmetric function $%
S_{n}(x_{1}+x_{2})$ is defined by
\begin{equation}
S_{n}(A)=S_{n}(x_{1}+x_{2})=\frac{x_{1}^{n+1}-x_{2}^{n+1}}{x_{1}-x_{2}},%
\text{\ for }n\geq 0. \label{66}
\end{equation}%
\end{definition}
\begin{remark} 
Given an alphabet $A=\{x_{1},x_{2}\}$, then for any positive integer $n$, we
have
\begin{equation}
S_{n}(x_{1}+x_{2})=h_{2}(x_{1},x_{2}).
\end{equation}
\end{remark}
\begin{definition}
 Let $n$ be a positive integer and $A=\{ x_{1},x_{2}\}
\ $ be a set of given variables. Then, the $n^{th}$ symmetric function $%
\phi_{n}(x_{1}+x_{2})$ is defined by
\begin{equation}
\phi_{n}(A)=\phi_{n}(x_{1}+x_{2})=x_{1}^{n}+x_{2}^{n},%
\text{\ for }n\geq 0. \label{66}
\end{equation}%
\end{definition}
\begin{remark}
Given an alphabet $A=\{x_{1},x_{2}\}$, then for any positive integer $n$, we
have
\begin{equation}
\phi_{n}(x_{1}+x_{2})=p_{2}(x_{1},x_{2}).
\end{equation}
\end{remark}
\begin{lemma}\cite{Maa}
Let $n$ be a positive integer. The following equalities hold
\begin{equation}
S_{n}(x_{1}+x_{2})-x_{1}x_{2}S_{n-2}(x_{1}+x_{2})=\phi_{n}(x_{1}+x_{2}).
\label{p}
\end{equation}%
\begin{equation}
\frac{1}{2}\left(
\phi_{n}(x_{1}+x_{2})+(x_{1}-x_{2})S_{n-1}(x_{1}+x_{2})\right) =x_{1}^{n}.
\end{equation}%
\begin{equation}
\frac{1}{2}\left(
\phi_{n}(x_{1}+x_{2})-(x_{1}-x_{2})S_{n-1}(x_{1}+x_{2})\right) =x_{2}^{n}.
\label{kll}
\end{equation}
\end{lemma}
\begin{lemma}\cite{Maa}
The exponential generating functions of $S_{n-1}(x_{1}+x_{2})$ and
$\phi_{n}(x_{1}+x_{2})$ are respectively
given by
\begin{equation}
\sum\limits_{n=0}^{\infty }S_{n-1}(x_{1}+x_{2})\frac{z^{n}}{n!}=\frac{1}{e_{1}-e_{2}}%
\left( \exp (x_{1}z)-\exp (x_{2}z)\right). \label{kl}
\end{equation}%
\begin{equation}
\sum\limits_{n=0}^{\infty }\phi_{n}(x_{1}+x_{2})\frac{%
z^{n}}{n!}=\exp (x_{1}z)+\exp (x_{2}z).  \label{l}
\end{equation}
\end{lemma}

Now, we present the definitions of the bivariate Fibonacci, bivariate Lucas, bivariate Balancing, and bivariate Lucas-Balancing polynomials, along with their respective Binet formulas (see \cite{Asci, Aşcı1, Catarino}).
\begin{definition}
The bivariate Fibonacci $\{F_{n}(y,t)\}_{n\in \mathbb{N}}$ and bivariate Lucas $\{L_{n}(y,t)\}_{n\in \mathbb{N}}$ polynomials are defined by the
following recurrence relations:
\begin{equation*}
\begin{cases}
F_{n}(y,t)=yF_{n-1}(y,t)+tF_{n-2}(y,t), \text{  } n\geqslant 2\\
F_{0}(y,t)=0, \text{ } F_{1}(y,t)=1
\end{cases}. \label{jio}
\end{equation*}
\begin{equation*}
\begin{cases}
L_{n}(y,t)=yL_{n-1}(y,t)+tL_{n-2}(y,t), \text{ } n\geqslant 2\\
L_{0}(y,t)=2, \text{  } L_{1}(y,t)=y
\end{cases}.
\end{equation*}
\end{definition}
Special cases of these bivariate polynomials are Fibonacci polynomials $F_{n}(y,1)$,  Lucas polynomials $L_{n}(y,1)$, Fibonacci  numbers $F_{n}(1,1)$, Lucas numbers $L_{n}(1,1)$.

The Binet formulas for  bivariate Fibonacci, bivariate Lucas are given as follows:
\begin{equation*}
F_{n}(y,t)=\frac{\lambda_{1}^{n}-\lambda_{2}^{n}}{\lambda_{1}-\lambda_{2}},
\end{equation*}
\begin{equation*}
L_{n}(y,t)=\lambda_{1}^{n}+\lambda_{2}^{n} ,
\end{equation*}
where
$\lambda_{1}=\frac{y+\sqrt{y^{2}+4t}}{2}$
and
$\lambda_{1}=\frac{y+\sqrt{y^{2}+4t}}{2}$ are the roots of the characteristic equation $\lambda^{2}-y\lambda-t=0$.
\begin{definition}
The bivariate balancing  $\{B^{*}_{n}(y,t)\}_{n\in \mathbb{N}}$ and bivariate Lucas-balancing $\{C_{n}(y,t)\}_{n\in \mathbb{N}}$ polynomials are defined by the
following recurrence relations:
\begin{equation*}
\begin{cases}
B^{*}_{n}(y,t)=6yB^{*}_{n-1}(y,t)-tB^{*}_{n-2}(y,t),\text{ } n\geqslant 2 \\
B^{*}_{0}(y,t)=0, \text{ } B^{*}_{1}(y,t)=1
\end{cases}. \label{khj}
\end{equation*}
\begin{equation*}
\begin{cases}
C_{n}(y,t)=6yC_{n-1}(y,t)-tC_{n-2}(y,t), \text{ } n\geqslant 2\\
C_{0}(y,t)=1, \text{ } C_{1}(y,t)=3y
\end{cases}.
\end{equation*}

\end{definition}
Notable special cases of these bivariate polynomials are balancing  polynomials $B^{*}_{n}(y,1)$,  Lucas-balancing  polynomials $C_{n}(y,1)$, balancing   numbers $B^{*}_{n}(1,1)$, Lucas-balancing  numbers $C_{n}(1,1)$.

The Binet formulas for  bivariate balancing , bivariate Lucas-balancing  are given as follows:
\begin{equation*}
B^{*}_{n}(y,t)=\frac{\lambda_{1}^{n}-\lambda_{2}^{n}}{\lambda_{1}-\lambda_{2}},
\end{equation*}
\begin{equation*}
C_{n}(y,t)=\frac{1}{2}(\lambda_{1}^{n}-\lambda_{2}^{n}),
\end{equation*}
with:
$\lambda_{1}=3y+\sqrt{9y^{2}-t}$
and
$\lambda_{2}=3y+\sqrt{9y^{2}-t}$ are the roots of the characteristic equation $$\lambda^{2}-6y\lambda+t=0$$.

On the other side, the Bernoulli and Euler numbers (resp. polynomials) $\left\{ B_{n}\right\} _{n\geqslant 0}$ and $%
\left\{ E_{n}\right\} _{n\geqslant 0}$ (resp. $\left\{ B_{n}(x)\right\} _{n\geqslant
0} $ and $\left\{ E_{n}(x)\right\} _{n\geqslant 0}$) are, respectively, defined by the following
exponential generating functions  (\cite{Dilcher, DilhO})
\begin{equation}
\sum\limits_{n=0}^{\infty}B_{n}\frac{z^{n}}{n!}=\frac{z}{\exp (z)-1},  \label{15}
\end{equation}%
\begin{equation}
\sum\limits_{n=0}^{\infty}E_{n}\frac{z^{n}}{n!}=\frac{2\exp (z)}{\exp (2z)+1}.
\label{16}
\end{equation} 
\begin{equation}
\sum\limits_{n=0}^{\infty}B_{n}(x)\frac{z^{n}}{n!}=\frac{z\exp (xz)}{\exp
(z)-1},  \label{17}
\end{equation}%
\begin{equation}
\sum\limits_{n=0}^{\infty}E_{n}(x)\frac{z^{n}}{n!}=\frac{2\exp (xz)}{\exp
(z)+1}.  \label{18}
\end{equation}

It is also defined Genocchi numbers (resp. polynomials) by the
exponential generating functions as follows(\cite{DilO})
\begin{equation}
\sum\limits_{n=0}^{\infty}G_{n}\frac{z^{n}}{n!}=\frac{2z}{\exp (z)+1}.
\label{116}
\end{equation} 
\begin{equation}
\sum\limits_{n=0}^{\infty }G_{n}(x)\frac{z^{n}}{n!}=\frac{2z\exp(xz)}{\exp(z)+1} \label{105}
\end{equation}

Convolutions play a central role in number theory, these relationships arise through sequences of important and distinct numbers and polynomials. It is usually obtained by manipulating the generating functions. Classic examples involving Fibonacci and Lucas numbers are given in the articles.(\cite{Lk, Lko}).

In view of the important  role of Bernoulli and Euler numbers and polynomials in mathematics, numerous studies have been conducted on the relationships between them and other well-known number sequences and polynomials. Several papers have revealed interesting results about them \cite{rev5,Ma, rev3, Frontczak2, rev6, rev4, rev1, rev2}.

In this paper, we present a more direct and systematic method that provides access to convolutions for second-order sequences and aims to find convolution sums of Bernoulli, Euler, and Genocchi numbers and polynomials, which we began constructing in the  article \cite{Maa}. The approach is accurate and efficient because it allows us to obtain various convolutions of second-order sequences using symmetric functions, without the need for iteration and complex calculations involving numbers or polynomials. The structure of the paper is as follows:

\textbf{ In section 2,} we presented new formulas for binomial convolution sequences related to symmetric functions, as well as convolution sums without the binomial symbol, at specific points enable us to establish some interesting
combinatorial identities involving bivariate polynomials.

\textbf{ In section 3,} 
 we give some new convolution sums of Bernoulli, Euler, and Genocchi numbers  with symmetric functions. From these
convolutions we deduce several new identities of Bernoulli, Euler and Genocchi numbers
and polynomials with special bivariate polynomials.

\textbf{ In section 4,} 
 We present new convolution sums involving Bernoulli, Euler, and Genocchi polynomials with symmetric functions. These convolutions allow us to establish several new identities for Bernoulli and Euler numbers and polynomials in connection with special bivariate polynomials.

\section{\textbf{Main results of convolution sums of symmetric functions}}

In this section, we first prove several new theorems about symmetric functions, and then we examine some special cases of bivariate polynomials.

\subsection{\textbf{New theorems}}

In this part, we are now in a position to provide two new theorems.
\begin{theorem}\label{th3}
Given an alphabet $A=\{x_{1},x_{2}\}$, then for any positive integer $n$, we
have
\begin{gather}
\sum_{n=0}^{n}\binom{n}{k}S_{n-k-1}(x_{1}+x_{2})S_{k-1}(x_{1}+x_{2})=\frac{1}{(x_{1}-x_{2})^{2}}\left( 2^{n}\phi_{n}(x_{1}+x_{2})-2(x_{1}+x_{2})^{n}\right).\label{xx}\\ 
\sum_{k=0}^{n}\binom{n}{k}\phi_{n-k}(x_{1}+x_{2})\phi_{k}(x_{1}+x_{2})=2^{n}\phi_{n}(x_{1}+x_{2})+2(x_{1}+x_{2})^{n}.\label{a2}\\
\sum_{k=0}^{n}\binom{n}{k}S_{n-k-1}(x_{1}+x_{2})\phi_{k}(x_{1}+x_{2})=2^{n}S_{n-1}(x_{1}+x_{2}).\label{a3}
\end{gather}
\end{theorem}
\begin{proof}
We square both sides of the exponential generating function associated with the symmetric function $S_{k-1}(x_{1}+x_{2})$, we obtain 
\begin{equation*}
\left(\sum_{k=0}^{\infty}S_{n-1}(x_{1}+x_{2})\frac{z^{n}}{n!}\right)\left(\sum_{k=0}^{\infty}S_{k-1}(x_{1}+x_{2})\frac{z^{n}}{n!}\right)=\sum_{n=0}^{\infty}\sum_{k=0}^{n}\binom{n}{k}S_{n-k-1}(x_{1}+x_{2})S_{k-1}(x_{1}+x_{2})\frac{z^{n}}{n!},
\end{equation*}
and
\begin{align*}
\left(\frac{1}{x_{1}-x_{2}}(\exp(x_{1}z)-\exp(x_{2}z))\right)^{2}&=\frac{1}{(x_{1}-x_{2})^{2}}\sum_{n=0}^{\infty} \left(2^{n}(x_{1}^{n}+x_{2}^{n})-2(x_{1}+x_{2})^{n}\right)\frac{z^{n}}{n!}\\
&=\frac{1}{(x_{1}-x_{2})^{2}}\sum_{n=0}^{\infty} \left(2^{n}\phi_{n}(x_{1}+x_{2})-2(x_{1}+x_{2})^{n}\right)\frac{z^{n}}{n!}.
\end{align*}
Comparing of the coefficients of $\frac{z^{n}}{n!}$, we obtain
\begin{equation*}
\sum_{k=0}^{n}\binom{n}{k}S_{n-k-1}(x_{1}+x_{2})S_{k-1}(x_{1}+x_{2})=\frac{1}{(x_{1}+x_{2})^{2}}\left( 2^{n}\phi_{n}(x_{1}+x_{2})-2(x_{1}+x_{2})^{n}\right).
\end{equation*}

Similarly, we square both sides of the exponential generating function associated with the symmetric function $\phi_{n}(x_{1}+x_{2})$ we obtain the equation \eqref{a2}. As for the relationship it is the result of the product of the exponential generating function of symmetric function $S_{n-1}(x_{1}+x_{2})$ and the exponential generating function of symmetric function $\phi_{n}(x_{1}+x_{2})$.

\end{proof}

\begin{theorem}\label{th5}
Given an alphabet $A=\{x_{1},x_{2}\}$, then for any positive integer $n$, we
have
\begin{gather}
\sum_{k=0}^{n}\phi_{k}(x_{1}+x_{2})\phi_{n-k}(x_{1}+x_{2})=(n+1)\phi_{n}(x_{1}+x_{2})+2S_{n}(x_{1}+x_{2}).\label{nnn}\\
\sum_{k=0}^{n}S_{k-1}(x_{1}+x_{2})S_{n-k-1}(x_{1}+x_{2})=\frac{1}{(x_{1}-x_{2})^{2}}\left((n+1)\phi_{n}(x_{1}+x_{2})-2S_{n}(x_{1}+x_{2})\right). \label{bc}\\
\sum_{k=0}^{n}\phi_{k}(x_{1}+x_{2})S_{n-k-1}(x_{1}+x_{2})=(n+1)S_{n-1}(x_{1}+x_{2}).\label{cvv}
\end{gather}
\end{theorem}
\begin{proof}
We prove only the first relation; the proof for the other relations follows in the same manner using the definition of the two symmetric functions $S_{n-1}(x_{1}+x_{2})$ and $\phi_{n}(x_{1}+x_{2})$.
\begin{align*}
\sum_{k=0}^{n}\phi_{k}(x_{1}+x_{2})\phi_{n-k}(x_{1}+x_{2})&=\sum_{k=0}^{n}(x_{1}^{k}+x_{2}^{k})(x_{1}^{n-k}+x_{2}^{n-k})\\
&=\sum_{k=0}^{n}(x_{1}^{n}+x_{1}^{k}x_{2}^{n-k}+x_{2}^{k}x_{1}^{n-k}+x_{2}^{n})\\
&=\sum_{k=0}^{n}(x_{1}^{n}+x_{2}^{n})+x_{2}^{n}\sum_{k=0}^{n}\left(\frac{x_{1}}{x_{2}}\right)^{k} +x_{1}^{n}\sum_{k=0}^{n}\left(\frac{x_{2}}{x_{1}}\right)^{k}\\
&=(n+1)(x_{1}^{n}+x_{2}^{n})+x_{2}^{n}\frac{\left(\frac{x_{1}}{x_{2}}\right)^{n+1}-1}{\frac{x_{1}}{x_{2}}-1}+x_{1}^{n}\frac{\left(\frac{x_{2}}{x_{1}}\right)^{n+1}-1}{\frac{x_{2}}{x_{1}}-1}\\
&=(n+1)(x_{1}^{n}+x_{2}^{n})+2\frac{x_{1}^{n+1}+x_{2}^{n+1}}{x_{1}-x_{2}}\\
&=(n+1)\phi_{n}(x_{1}+x_{2})+2S_{n}(x_{1}+x_{2}).
\end{align*}

The proof is now complete.
\end{proof}

\subsection{\textbf{Some applications}}

We, now consider the Theorems \ref{th3} and \ref{th5} to derive the following two cases.
\begin{enumerate}
	\item [\textbf{Case 1.}]
Let $x_{1}=\frac{y+\sqrt{y^{2}+4t}}{2}$ and $x_{2}=\frac{y-\sqrt{y^{2}+4t}}{2}$, we obtain the results of binomial convolution sums and convolution sums without the binomial symbol involving for sequences bivariate Fibonacci,  bivariate Lucas  in the following results. 

\begin{corollary}\label{cr1}
Let $n$ be a positive integer, we have
\begin{gather*}
\sum_{k=0}^{n}\binom{n}{k}F_{n-k}(y,t)F_{k}(y,t)=\frac{1}{y^{2}+4t}\left(2^{n}L_{n}(y,t)-2y^{n}\right). \label{q}\\
\sum_{k=0}^{n}\binom{n}{k}L_{n-k}(y,t)L_{k}(y,t)=2^{n}F_{n}(y,t)+2y^{n}.\\
\sum_{k=0}^{n}\binom{n}{k}F_{n-k}(y,t)L_{k}(y,t)=2^{n}F_{n}(y,t).
\end{gather*}
\end{corollary}

\begin{corollary}\label{cr3}
Let $n$ be a positive integer, we have
\begin{gather*}
\sum_{k=0}^{n}F_{n-k}(y,t)F_{k}(y,t)=\frac{1}{y^{2}+4t}\left((n+1)L_{n}(y,t)-2F_{n+1}(y,t)\right).\\
\sum_{k=0}^{n}L_{n-k}(y,t)L_{k}(y,t)=(n+1)L_{n}(y,t)+2F_{n+1}(y,t).\\
\sum_{k=0}^{n}F_{n-k}(y,t)L_{k}(y,t)=(n+1)F_{n}(y,t).
\end{gather*}
\end{corollary}
\begin{itemize}
\item If we put $t=1$ in the corollaries \ref{cr1} and \ref{cr3}, the results for binomial convolution sums and convolution sums without the binomial symbol, involving  Fibonacci and Lucas polynomials, are obtained.
\item If we set $y=t=1$ in the corollaries \ref{cr1} and \ref{cr3}, we derive the results for binomial convolution sums as well as convolution sums without the binomial symbol, involving  Fibonacci and Lucas numbers(see \cite{Lk, Lko}).
\end{itemize}
\item [\textbf{Case 2.}]
Let $x_{1}=3y+\sqrt{9y^{2}-t}$ and $x_{2}=3y-\sqrt{9y^{2}-t}$, then we have we obtain the results of binomial convolution sums of bivariate balancing and  bivariate Lucas-balancing polynomials in the the following  corollary.

\begin{corollary}\label{cr4}
Let $n$ be a positive integer, we have
\begin{gather*}
\sum_{k=0}^{n}\binom{n}{k}B^{*}_{n-k}(y,t)B^{*}_{k}(y,t)=\frac{1}{2(9y^{2}-t)}\left(2^{n}C_{n}(y,t)-(6y)^{n}\right). \label{nb}\\
\sum_{k=0}^{n}\binom{n}{k}2C_{n-k}(y,t)C_{k}(y,t)=2^{n}C_{n}(y,t)+(6y)^{n}.\\
\sum_{k=0}^{n}\binom{n}{k}B^{*}_{n-k}(y,t)C_{k}(y,t)=2^{n-1}B^{*}_{n}(y,t).
\end{gather*}
\end{corollary}
 Convolution sums without the binomial symbol, involving bivariate balancing and bivariate Lucas-balancing polynomials, are given by the following corollary.  
\begin{corollary}\label{cr6}
Let $n$ be a positive integer, we have
\begin{gather*}
\sum_{k=0}^{n}B^{*}_{n-k}(y,t)B^{*}_{k}(y,t)=\frac{1}{2(9y^{2}-t)}\left((n+1)C_{n}(y,t)-B^{*}_{n+1}(y,t)\right).\\
\sum_{k=0}^{n}2C_{n-k}(y,t)C_{k}(y,t)=(n+1)C_{n}(y,t)+B^{*}_{n+1}(x,y).\\
\sum_{k=0}^{n}2B^{*}_{n-k}(y,t)C_{k}(y,t)=(n+1)B^{*}_{n}(y,t).
\end{gather*}
\end{corollary}
\begin{itemize}
\item If we put $t=1$ in the corollaries \ref{cr4} and \ref{cr6}, the results for binomial convolution sums and convolution sums without the binomial symbol, involving  balancing and Lucas-balancing polynomials, are obtained.
\item If we are setting $y=t=1$ in the corollaries \ref{cr4} and \ref{cr6}, we derive the results for binomial convolution sums as well as convolution sums without the binomial symbol, involving  balancing and Lucas-balancing numbers.
\end{itemize}
\end{enumerate}
\section{\textbf{New convolution of bivariate polynomials and known numbers }}
By using the Bernoulli  Euler and Genochi numbers and the
symmetric functions we prove some new theorems and we give some
applications.

\subsection{\textbf{New theorems}}

In this part, we are in a position to provide some new theorems.
\begin{theorem}\label{th6}
Given an alphabet $A=\{x_{1},x_{2}\}$, then for any positive integer $n$, we
have
\begin{equation}
\sum_{\underset{n\equiv k[2]}{k=0}}^{n}\binom{n}{k}(x_{1}-x_{2})^{n-k}\left(2^{k}\phi_{k}(x_{1}+x_{2})+2(x_{1}+x_{2})^{k}\right)G_{n-k}=-n2^{n-1}(x_{1}-x_{2})^{2}S_{n-2}(x_{1}+x_{2}). \label{qxa}
\end{equation}
\end{theorem}
\begin{proof}
We use the change of variable $z=(x_{1}-x_{2})z$ in 
\eqref{116}, we obtain
\begin{equation}
\sum_{n=0}^{\infty}(x_{1}-x_{2})^{n}G_{n}\frac{z^{n}}{n!}=\frac{2(x_{1}-x_{2})z}{\exp((x_{1}-x_{2})z+1)}. \label{d}
\end{equation}
Multiply relation \eqref{d} by the square of the exponential generating function associated with the symmetric function $\phi_{k}(x_{1}+x_{2})$, we get 
\begin{align*}
&\left(\sum_{k=0}^{\infty}\phi_{k}(x_{1}+x_{2})\frac{z^{k}}{k!}\right)^{2}\left(\sum_{n=0}^{\infty}(x_{1}-x_{2})^{n}G_{n}\frac{z^{n}}{n!}\right)\\
&=\left(\sum_{k=0}^{\infty}\left(2^{k}\phi_{k}(x_{1}+x_{2})+2(x_{1}+x_{2})^{k}\right) \frac{z^{k}}{k!}\right)\left(\sum_{n=0}^{\infty}(x_{1}-x_{2})^{n}G_{n}\frac{z^{n}}{n!}\right)\\
&=\sum_{n=0}^{\infty}\sum_{k=0}^{n}\binom{n}{k}(x_{1}-x_{2})^{n-k}\left(2^{k}\phi_{k}(x_{1}+x_{2})+2(x_{1}+x_{2})^{k}\right)G_{n-k}\frac{z^{n}}{n!},
\end{align*}
and
\begin{align*}
\exp(2x_{2}z)(\exp((x_{1}-x_{2})z)+1)^{2}\frac{2(x_{1}-x_{2})z}{\exp((x_{1}-x_{2})z)+1}&=2(x_{1}-x_{2})z(\exp((x_{1}+x_{2})z)+\exp(2x_{2}z))\\
&=\sum_{n=0}^{\infty}2n(x_{1}-x_{2})\left((x_{1}+x_{2})^{n-1}+(2x_{2})^{n-1}\right)\frac{z^{n}}{n!}.
\end{align*}
By comparing the coefficients of 
 $\frac{z^{n}}{n!}$, we obtain
\begin{equation*}
\sum_{k=0}^{n}\binom{n}{k}(x_{1}-x_{2})^{n-k}\left(2^{k}\phi_{k}(x_{1}+x_{2})+2(x_{1}+x_{2})^{k}\right)G_{n-k}=2n(x_{1}-x_{2})\left((x_{1}+x_{2})^{n-1}+(2x_{2})^{n-1}\right),
\end{equation*}
We know that $G_{0}=0$, $G_{1}=1$ and for $k\geqslant 0$, $G_{2k+1}=0$. By applying the equation  \eqref{kll}, we obtain
\begin{equation*}
\sum_{\underset{n\equiv k[2]}{k=0}}^{n}\binom{n}{k}(x_{1}-x_{2})^{n-k}\left(2^{k}\phi_{k}(x_{1}+x_{2})+2(x_{1}+x_{2})^{k}\right)G_{n-k}=-n2^{n-1}(x_{1}-x_{2})^{2}S_{n-2}(x_{1}+x_{2}).
\end{equation*}

The proof is completed.
\end{proof}
\begin{theorem}
Given an alphabet $A=\{x_{1},x_{2}\}$, then for any positive integer $n$, we
have
\begin{equation}
\sum_{k=0}^{n}\binom{n}{k}(x_{1}-x_{2})^{n-k}(n-k-1)\left(2^{k}\phi_{k}(x_{1}+x_{2})+2(x_{1}+x_{2})^{k}\right)G_{n-k}=-2n(n-1)(x_{1}-x_{2})^{2}(x_{1}+x_{2})^{n-2}.\label{qxaa1}
\end{equation}
\end{theorem}
\begin{proof}
We divide both sides of the equation \eqref{d} by $z$ then we differentiate with respect to $z$, we have
\begin{equation*}
\sum_{n=0}^{\infty}(n-1)(x_{1}-x_{2})^{n}G_{n}\frac{z^{n-2}}{n!}=\frac{-2(x_{1}-x_{2})^{2}\exp((x_{1}-x_{2})z)}{(\exp((x_{1}-x_{2})z)+1)^{2}}.
\end{equation*}
We multiply both sides by $z^{2}$, We get
\begin{equation}
\sum_{n=0}^{\infty}(n-1)(x_{1}-x_{2})^{n}G_{n}\frac{z^{n}}{n!}=\frac{-2(x_{1}-x_{2})^{2}z^{2}\exp((x_{1}-x_{2})z)}{(\exp((x_{1}-x_{2})z)+1)^{2}}.\label{ss3}
\end{equation}
Multiply the equation \eqref{ss3} by the square of the exponential generating function related to the symmetric function $\phi_{k}(x_{1}+x_{2})$, we derive
\begin{align*}
&\left(\sum_{k=0}^{\infty}\phi_{k}(x_{1}+x_{2})\frac{z^{k}}{k!}\right)^{2}\left(\sum_{n=0}^{\infty}(n-1)(x_{1}-x_{2})^{n}G_{n}\frac{z^{n}}{n!}\right)\\
&=\left(\sum_{k=0}^{\infty}\left(2^{k}\phi_{k}(x_{1}+x_{2})+2(x_{1}+x_{2})^{k}\right) \frac{z^{k}}{k!}\right)\left(\sum_{n=0}^{\infty}(n-1)(x_{1}-x_{2})^{n}G_{n}\frac{z^{n}}{n!}\right)\\
&=\sum_{n=0}^{\infty}\sum_{k=0}^{n}\binom{n}{k}(x_{1}-x_{2})^{n-k}(n-k-1)\left(2^{k}\phi_{k}(x_{1}+x_{2})+2(x_{1}+x_{2})^{k}\right)G_{n-k}\frac{z^{n}}{n!},
\end{align*}
On the other side, we have
\begin{align*}
\exp(2x_{2}z)(\exp((x_{1}-x_{2})z)+1)^{2}\frac{-2(x_{1}-x_{2})^{2}z^{2}\exp((x_{1}-x_{2})z)}{(\exp((x_{1}-x_{2})z)+1)^{2}}&=-2(x_{1}-x_{2})^{2}z^{2}\exp((x_{1}+x_{2})z)\\
&=\sum_{n=0}^{\infty}-2n(n-1)(x_{1}-x_{2})^{2}(x_{1}+x_{2})^{n-2}\frac{z^{n}}{n!}.
\end{align*}
Comparing of the coefficients of $\frac{z^{n}}{n!}$, we obtain the desired result.
\end{proof}
\begin{theorem}
Given an alphabet $A=\{x_{1},x_{2}\}$, then for any positive integer $n$, we
have
\begin{equation}
\sum_{\underset{n\equiv k[2]}{k=0}}^{m}\binom{n}{k}(x_{1}-x_{2})^{m-k}\frac{G_{m-k-2}}{m-k-2}\left(2^{k}\phi_{k}(x_{1}+x_{2})+2(x_{1}+x_{2})^{k}\right)=-2(x_{1}-x_{2})^{2}(x_{1}+x_{2})^{m}.\label{56}
\end{equation}
\end{theorem}
\begin{proof}
When putting
$n=m+2$,
we remark
$$\binom{n}{k}(n-k-1)=\binom{m}{k}\frac{(m+1)(m+2)}{m-k+2}.$$

Then switching the order of the two opposite terms, 
$k=m+1$,
$k=m+2$
on the left side of the relation 
\eqref{qxaa1}, by simplifying, we arrive at the desired result.
\end{proof}

\begin{theorem}
Given an alphabet $A=\{x_{1},x_{2}\}$, then for any positive integer $n$, we
have
\begin{gather}
\sum_{\underset{n\equiv k[2]}{k=0}}^{n}\binom{n}{k}(x_{1}-x_{2})^{n-k}\left(2^{k}\phi_{k}(x_{1}+x_{2})-2(x_{1}+x_{2})^{k}\right)B_{n-k}=n2^{n-2}(x_{1}-x_{2})^{2}S_{n-2}(x_{1}+x_{2}).  \\
\sum_{\underset{n\equiv k[2]}{k=0}}^{n}\binom{n}{k}(x_{1}-x_{2})^{n-k}(1-2^{n-k})\left(2^{k}\phi_{k}(x_{1}+x_{2})+2(x_{1}+x_{2})^{k}\right)B_{n-k}=-n2^{n-2}(x_{1}-x_{2})^{2}S_{n-2}(x_{1}+x_{2}).\label{a98a}
\end{gather}
\end{theorem}
\begin{proof}
We use the change of variable $z=(x_{1}-x_{2})z$ in 
\eqref{15}, we obtain
\begin{equation}
\sum_{n=0}^{\infty}(x_{1}-x_{2})^{n}B_{n}\frac{z^{n}}{n!}=\frac{(x_{1}-x_{2})z}{\exp ((x_{1}-x_{2})z)-1}. \label{g}
\end{equation}

Take the square of the exponential generating function associated with the symmetric function $S_{k-1}(x_{1}+x_{2})$ and multiply it by relation \eqref{g}. You will arrive at

\begin{align*}
&\left(\sum_{k=0}^{\infty}S_{k-1}(x_{1}+x_{2})\frac{z^{k}}{k!}\right)^{2}\left(\sum_{n=0}^{\infty}(x_{1}-x_{2})^{n}B_{n}\frac{z^{n}}{n!}\right)\\
&=\left(\sum_{k=0}^{\infty}\frac{1}{(x_{1}-x_{2})^{2}}\left(2^{k}\phi_{k}(x_{1}+x_{2})-2(x_{1}+x_{2})^{k}\right) \frac{z^{k}}{k!}\right)\left(\sum_{n=0}^{\infty}(x_{1}-x_{2})^{n}B_{n}\frac{z^{n}}{n!}\right)\\
&=\sum_{n=0}^{\infty}\sum_{k=0}^{n}\binom{n}{k}(x_{1}-x_{2})^{n-k-2}\left(2^{k}\phi_{k}(x_{1}+x_{2})-2(x_{1}+x_{2})^{k}\right)B_{n-k}\frac{z^{n}}{n!},
\end{align*}
and
\begin{align*}
\frac{\exp(2x_{2}z)}{(x_{1}-x_{2})^{2}}(\exp((x_{1}-x_{2})z)-1)^{2}\frac{(x_{1}-x_{2})z}{\exp((x_{1}-x_{2})z)-1}&=\frac{z}{x_{1}-x_{2}}(\exp((x_{1}+x_{2})z)-\exp(2x_{2}z))\\
&=\sum_{n=0}^{\infty}\frac{n}{x_{1}-x_{2}}\left((x_{1}+x_{2})^{n-1}-(2x_{2})^{n-1}\right)\frac{z^{n}}{n!}.
\end{align*}
By equating the coefficients of $\frac{z^{n}}{n!}$, we obtain
\begin{equation*}
\sum_{k=0}^{n}\binom{n}{k}(x_{1}-x_{2})^{n-k-2}\left(2^{k}\phi_{k}(x_{1}+x_{2})-2(x_{1}+x_{2})^{k}\right)B_{n-k}=\frac{n}{x_{1}-x_{2}}\left((x_{1}+x_{2})^{n-1}-(2x_{2})^{n-1}\right).
\end{equation*}

We have:
$B_{0}=1$ ,$B_{1}=\frac{-1}{2}$ 
for
 $k>0$, 
$B_{2k+1}=0$
and by applying the relationship \eqref{kll}, we obtain
\begin{equation*}
\sum_{\underset{n\equiv k[2]}{k=0}}^{n}\binom{n}{k}(x_{1}-x_{2})^{n-k}\left(2^{k}\phi_{k}(x_{1}+x_{2})-2(x_{1}+x_{2})^{k}\right)B_{n-k}=n2^{n-2}(x_{1}-x_{2})^{2}S_{n-2}(x_{1}+x_{2}).
\end{equation*}

To prove the relation \eqref{a98a}, we apply the property 
$G_{n}=2(1-2^{n})B_{n}$
to the relation \eqref{qxa}.
\end{proof}
\begin{theorem}
Given an alphabet $A=\{x_{1},x_{2}\}$, then for any positive integer $n$, we
have
\begin{gather}
\sum_{k=0}^{n}\binom{n}{k}(x_{1}-x_{2})^{n-k}(n-k-1)\left(2^{k}\phi_{k}(x_{1}+x_{2})-2(x_{1}+x_{2})^{k}\right)B_{n-k}=n(1-n)(x_{1}-x_{2})^{2}(x_{1}+x_{2})^{n-2}. \label{fvd}\\
\sum_{k=0}^{n}\binom{n}{k}(x_{1}-x_{2})^{n-k}(n-k-1)\left(2^{k}\phi_{k}(x_{1}+x_{2})+2(x_{1}+x_{2})^{k}\right)B_{n-k}=-n(1-n)(x_{1}-x_{2})^{2}(x_{1}+x_{2})^{n-2}.\label{dw}
\end{gather}
\end{theorem}
\begin{proof}
We divide both sides of the equation \eqref{g} by $z$ then we differentiate with respect to $z$, we have
\begin{equation*}
\sum_{n=0}^{\infty}(n-1)(x_{1}-x_{2})^{n}B_{n}\frac{z^{n-2}}{n!}=\frac{-(x_{1}-x_{2})^{2}\exp((x_{1}-x_{2})z)}{(\exp((x_{1}-x_{2})z)-1)^{2}},
\end{equation*}
 and we multiply both sides by $z^{2}$, We get
\begin{equation}
\sum_{n=0}^{\infty}(n-1)(x_{1}-x_{2})^{n}B_{n}\frac{z^{n}}{n!}=\frac{-(x_{1}-x_{2})^{2}z^{2}\exp((x_{1}-x_{2})z)}{(\exp((x_{1}-x_{2})z)-1)^{2}}.\label{ss33}
\end{equation}

Multiply the equation \eqref{ss33} by the square of the exponential generating function related to the symmetric function $S_{k-1}(x_{1}+x_{2})$, we derive
\begin{align*}
&\left(\sum_{k=0}^{\infty}S_{k-1}(x_{1}+x_{2})\frac{z^{k}}{k!}\right)^{2}\left(\sum_{n=0}^{\infty}(n-1)(x_{1}-x_{2})^{n}B_{n}\frac{z^{n}}{n!}\right)\\
&=\left(\sum_{k=0}^{\infty}\frac{1}{(x_{1}-x_{2})^{2}}\left(2^{k}\phi_{k}(x_{1}+x_{2})-2(x_{1}+x_{2})^{k}\right) \frac{z^{k}}{k!}\right)\left(\sum_{n=0}^{\infty}(n-1)(x_{1}-x_{2})^{n}B_{n}\frac{z^{n}}{n!}\right)\\
&=\sum_{n=0}^{\infty}\sum_{k=0}^{n}\binom{n}{k}(n-k-1)(x_{1}-x_{2})^{n-k-2}\left(2^{k}\phi_{k}(x_{1}+x_{2})-2(x_{1}+x_{2})^{k}\right)B_{n-k}\frac{z^{n}}{n!},
\end{align*}
then
\begin{align*}
\frac{\exp(2x_{2}z)}{(x_{1}-x_{2})^{2}}(\exp((x_{1}+x_{2})z)-1)^{2}\frac{-(x_{1}-x_{2})^{2}z^{2}\exp((x_{1}-x_{2})z)}{(\exp((x_{1}-x_{2})z)-1)^{2}}&=-z^{2}\exp((x_{1}+x_{2})z)\\
&=\sum_{n=0}^{\infty}n(1-n)(x_{1}+x_{2})^{n-2}\frac{z^{n}}{n!}.
\end{align*}
Comparing of the coefficients of $\frac{z^{n}}{n!}$, we obtain the desired result.

To prove the relation \eqref{dw}, we apply the property 
$G_{n}=2(1-2^{n})B_{n}$
to the relation \eqref{qxaa1}.
\end{proof}
\begin{theorem}
Given an alphabet $A=\{x_{1},x_{2}\}$, then for any positive integer $n$, we
have
\begin{align}
&\sum_{\underset{m\equiv k[2]}{k=0}}^{m}\binom{m}{k}(x_{1}-x_{2})^{m-k}\frac{B_{m-k-2}}{m-k-2}\left(2^{k}\phi_{k}(x_{1}+x_{2})-2(x_{1}+x_{2})^{k}\right) \notag \\
&=\frac{2^{m+2}\left(\phi_{m+2}(x_{1}+x_{2})-2(x_{1}+x_{2})^{m+2}\right)  }{(x_{1}-x_{2})^{2}(m+2)(m+1)}-2(x_{1}+x_{2})^{m}.\label{wcx}
\end{align}
\begin{gather}
\sum_{\underset{m\equiv k[2]}{k=0}}^{m}\binom{m}{k}(x_{1}-x_{2})^{m-k}(2^{m-k+2}-1)\frac{B_{m-k-2}}{m-k-2}\left(2^{k}\phi_{k}(x_{1}+x_{2})+2(x_{1}+x_{2})^{k}\right)=(e_{1}+e_{2})^{m}. \label{qm}
\end{gather}
\end{theorem}
\begin{proof}
When putting
$n=m+2$
We remark
$$\binom{n}{k}(n-k-1)=\binom{m}{k}\frac{(m+1)(m+2)}{m-k+2},$$
then switching the order of the two opposite terms, 
$k=m+1$,
$k=m+2$
on the left side of the relation 
\eqref{fvd}, we arrive at 
\begin{align*}
&\sum_{k=0}^{m}\binom{m}{k}(x_{1}-x_{2})^{m-k}\frac{(m+1)(m+2)B_{m-k-2}}{m-k-2}\left(2^{k}\phi_{k}(x_{1}+x_{2})-2(x_{1}+x_{2})^{k}\right)\\ &=2^{m+2}\left(\phi_{m+2}(x_{1}+x_{2})-2(x_{1}+x_{2})^{m+2}\right) \\
&-2(m+2)(m+1)(x_{1}+x_{2})^{m},
\end{align*}
then
\begin{align*}
&\sum_{\underset{m\equiv k[2]}{k=0}}^{m}\binom{m}{k}(x_{1}-x_{2})^{m-k}\frac{B_{m-k-2}}{m-k-2}\left(2^{k}\phi_{k}(x_{1}+x_{2})-2(x_{1}+x_{2})^{k}\right)  \\
&=\frac{2^{m+2}\left(\phi_{m+2}(x_{1}+x_{2})-2(x_{1}+x_{2})^{m+2}\right)  }{(x_{1}-x_{2})^{2}(m+2)(m+1)}-2(x_{1}+x_{2})^{m}.
\end{align*}

To prove relation \eqref{qm}, we apply the same method to relation  \eqref{dw}.
\end{proof}

\begin{theorem}\label{gj}
Given an alphabet $A=\{x_{1}, x_{2}\}$, then for any positive integer $n$, we
have
\begin{equation}
\sum_{\underset{n\equiv k[2]}{k=0}}^{n}\binom{n}{k}\left(\frac{x_{1}-x_{2}}{2}\right)^{n-k}\left(2^{k}\phi_{k}(x_{1}+x_{2})+2(x_{1}+x_{2})^{k}\right)E_{n-k}=2^{1-n}\left((3x_{1}+x_{2})^{n}+(x_{1}+3x_{2})^{n}\right). \label{sxv}
\end{equation}
\end{theorem}
\begin{proof}
We use the change of variable $z=\frac{e_{1}-e_{2}}{2}z$ in 
\eqref{16}, we obtain
\begin{equation}
\sum_{n=0}^{\infty}\left( \frac{x_{1}-x_{2}}{2}\right)^{n}E_{n}\frac{z^{n}}{n!}=\frac{2\exp(\frac{x_{1}-x_{2}}{2}z)}{\exp((x_{1}-x_{2})z)+1}. \label{h}
\end{equation}

Square the exponential generating function associated with the symmetric function  $\phi_{k-1}(x_{1}+x_{2})$ and multiply it by relation \eqref{h}. You will get to
\begin{align*}
&\left(\sum_{n=0}^{\infty}(\frac{x_{1}-x_{2}}{2})^{n}E_{n}\frac{z^{n}}{n!}\right)\left(\sum_{k=0}^{\infty}\phi_{k}(x_{1}+x_{2})\frac{z^{k}}{k!}\right)^{2}\\
&=\left(\sum_{n=0}^{\infty}(\frac{x_{1}-x_{2}}{2})^{n}E_{n}\frac{z^{n}}{n!}\right)\left(\sum_{k=0}^{\infty}(2^{k}\phi_{k}(x_{1}+x_{2})+2(x_{1}+x_{2})^{k})\frac{z^{k}}{k!}\right)\\
&=\sum_{n=0}^{\infty}\sum_{k=0}^{n}\binom{n}{k}(\frac{x_{1}-x_{2}}{2})^{n-k}\left(2^{k}\phi_{k}(x_{1}+x_{2})+2(2(x_{1}+x_{2}))^{k}\right)E_{n-k}\frac{z^{n}}{n!},
\end{align*}
then
\begin{align*}
\left(\exp(x_{1}z)+\exp(x_{2}z)\right)^{2}\frac{2\exp(\frac{x_{1}-x_{2}}{2})}{\exp((x_{1}-x_{2})z)+1}
&=2\left(\exp\left(\frac{3x_{1}+x_{2}}{2}z\right)+\exp\left(\frac{x_{1}+3x_{2}}{2}z\right)\right)\\
&=\sum_{n=0}^{\infty}2^{1-n}\left((3x_{1}+x_{2})^{n}+(x_{1}+3x_{2})^{n}\right)\frac{z^{n}}{n!}.
\end{align*}
Comparing the coefficients of $\frac{z^{n}}{n!}$, we obtain \begin{equation*}
\sum_{k=0}^{n}\binom{n}{k}\left(\frac{x_{1}-x_{2}}{2}\right)^{n-k}\left(2^{k}\phi_{k}(x_{1}+x_{2})+2(x_{1}+x_{2})^{k}\right)E_{n-k}=2^{1-n}\left((3x_{1}+x_{2})^{n}+(x_{1}+3x_{2})^{n}\right).
\end{equation*}

We have for $k>0$, $E_{2k+1}=0$, then we get the desired result.
\end{proof}
\subsection{\textbf{Some special cases}}

We now consider the Theorems \ref{th6}- \ref{gj} to derive the following two cases.
\begin{enumerate}
	\item [\textbf{Case 1.}]
Let $x_{1}=\frac{y+\sqrt{y^{2}+4t}}{2}$ and $x_{2}=\frac{y-\sqrt{y^{2}+4t}}{2}$, then we have the following the following combinatorial
results of Genocchi, Bernoulli and Euler  numbers with bivariate Fibonacci,  bivariate Lucas polynomials.
\begin{corollary}\label{1}
Let $n$ be a positive integer, we have
\begin{equation*}
\sum_{\underset{n\equiv k[2]}{k=0}}^{n}\binom{n}{k}\sqrt{y^{2}+4t}^{n-k}\left(2^{k}L_{k}(y,t)+2x^{k}\right)G_{n-k}=-n2^{n-1}(y^{2}+4t)F_{n-1}(y,t).
\end{equation*}
\end{corollary}

\begin{corollary}
Let $n$ be a positive integer, we have
\begin{equation*}
\sum_{k=0}^{n}\binom{n}{k}\sqrt{y^{2}+4t}^{n-k}(n-k-1)\left(2^{k}L_{k}(y,t)+2x^{k}\right)G_{n-k}=2n(1-n)(y^{2}+4t)x^{n-2}.
\end{equation*}
\end{corollary}

\begin{corollary}
Let $n$ be a positive integer, we have
\begin{equation*}
\sum_{\underset{m\equiv k[2]}{k=0}}^{m}\binom{m}{k}\sqrt{y^{2}+4t}^{m-k}\frac{G_{m-k-2}}{n-k-2}\left(2^{k}L_{k}(y,t)+2y^{k}\right)=-2(y^{2}+4t)x^{m}.
\end{equation*}
\end{corollary}
\begin{corollary}
Let $n$ be a positive integer, we have
\begin{gather*}
\sum_{k=0}^{n}\binom{n}{k}\sqrt{y^{2}+4t}^{n-k}\left(2^{k}L_{k}(y,t)-2y^{k}\right)B_{n-k}=n2^{n-2}(y^{2}+4t)L_{n-1}(y,t).\\
\sum_{k=0}^{n}\binom{n}{k}\sqrt{y^{2}+4t}^{n-k}(1-2^{n-k})\left(2^{k}L_{k}(y,t)+2y^{k}\right)B_{n-k}=-n2^{n-2}(y^{2}+4t)L_{n-1}(y,t).
\end{gather*}
\end{corollary}
\begin{corollary}
Let $n$ be a positive integer, we have
\begin{gather*}
\sum_{\underset{n\equiv k[2]}{k=0}}^{n}\binom{n}{k}\sqrt{y^{2}+4t}^{n-k}(n-k-1)\left(2^{k}L_{k}(y,t)-2y^{k}\right)B_{n-k}=n(1-n)(y^{2}+4t)^{2}x^{n-2}.\\
\sum_{\underset{n\equiv k[2]}{k=0}}^{n}\binom{n}{k}\sqrt{y^{2}+4t}^{n-k}(2^{n-k}-	1)(n-k-1)\left(2^{k}L_{k}(y,t)+2y^{k}\right)B_{n-k}=-n(1-n)(y^{2}+4t)^{2}y^{n-2}.
\end{gather*}
\end{corollary}
\begin{corollary}
Let $n$ be a positive integer, we have
\begin{gather*}
\sum_{\underset{n\equiv k[2]}{k=0}}^{n}\binom{n}{k}\sqrt{y^{2}+4t}^{n-k}\frac{B_{n-k-2}}{n-k-2}\left(2^{k}L_{k}(y,t)-2y^{k}\right)=\frac{2^{m+2}(L_{m+2}(y,t)-2x^{m+2})}{(y^{2}+4t)(m+2)(m+1)}-y^{m}.\\
\sum_{\underset{n\equiv k[2]}{k=0}}^{n}\binom{n}{k}(2^{n-k}-	1)\sqrt{y^{2}+4t}^{n-k}\frac{B_{n-k-2}}{n-k-2}\left(2^{k}L_{k}(y,t)+2y^{k}\right)=y^{m}.
\end{gather*}
\end{corollary}
\begin{corollary}\label{2}
Let $n$ be a positive integer, we have
\begin{equation*}
\sum_{\underset{n\equiv k[2]}{k=0}}^{n}\binom{n}{k}\left(\frac{\sqrt{y^{2}+4t}}{2}^{n-k}\right) \left(2^{k}L_{k}(y,t)+2y^{k}\right)E_{n-k}=2^{1-n}\left((2y+\sqrt{y^{2}+4t})^{n}+(2y-\sqrt{y^{2}+4t})^{n}\right).
\end{equation*}
\end{corollary}
\begin{itemize}
\item If we take $t=1$ in the corollaries \ref{1}$-$\ref{2}, we can   obtain the connections between Genocchi, Bernoulli and Euler  numbers and
 Fibonacci and Lucas polynomials.
\item By setting $y=t=1$ in the corollaries \ref{1}$-$\ref{2}, We derive connections between Genocchi, Bernoulli and Euler  numbers and
 Fibonacci and Lucas numbers.
\end{itemize}
\item [\textbf{Case 2.}]
Let $x_{1}=3x+\sqrt{9x^{2}-y}$ and $x_{2}=3x-\sqrt{9x^{2}-y}$, then we have the following corollaries of relations between Genocchi, Bernoulli and Euler  numbers with bivariate balancing,  bivariate Lucas-balancing polynomials. 
\begin{corollary}\label{3}
Let $n$ be a positive integer, we have
\begin{equation*}
\sum_{\underset{n\equiv k[2]}{k=0}}^{n}\binom{n}{k}(2\sqrt{9x^{2}-y})^{n-k}\left(2^{k+1}C_{k}(x,y)+2(6x)^{k}\right)G_{n-k}=-n2^{n+1}(9x^{2}-y)^{2}B^{*}_{n-1}(x,y).
\end{equation*}
\end{corollary}

\begin{corollary}
Let $n$ be a positive integer, we have
\begin{equation*}
\sum_{k=0}^{n}\binom{n}{k}(2\sqrt{9x^{2}-y})^{n-k}(n-k-1)\left(2^{k+1}C_{k}(x,y)+2(6x)^{k}\right)G_{n-k}=8n(1-n)(9x^{2}-y)(6x)^{n-2}.
\end{equation*}
\end{corollary}

\begin{corollary}
Let $n$ be a positive integer, we have
\begin{equation*}
\sum_{\underset{n\equiv k[2]}{k=0}}^{n}\binom{n}{k}(2\sqrt{9x^{2}-y})^{n-k}\frac{G_{n-k-2}}{n-k-2}\left(2^{k+1}C_{k}(x,y)+2(6x)^{k}\right)=-8(9x^{2}-y)(6x)^{n}.
\end{equation*}
\end{corollary}

\begin{corollary}
Let $n$ be a positive integer, we have
\begin{gather*}
\sum_{k=0}^{n}\binom{n}{k}(2\sqrt{9x^{2}-y})^{n-k}\left(2^{k}C_{k}(x,y)-(6x)^{k}\right)B_{n-k}=n2^{n-1}(9x^{2}-y)B^{*}_{n-1}(x,y).\\
\sum_{k=0}^{n}\binom{n}{k}(1-2^{n-k})\sqrt{9x^{2}-y}^{n-k}\left(2^{k}C_{k}(x,y)+(6x)^{k}\right)B_{n-k}=-n2^{n-1}(9x^{2}-y)B^{*}_{n-1}(x,y).
\end{gather*}
\end{corollary}

\begin{corollary}
Let $n$ be a positive integer, we have
\begin{gather*}
\sum_{\underset{n\equiv k[2]}{k=0}}^{n}\binom{n}{k}(2\sqrt{9x^{2}-y})^{n-k}(n-k-1)\left(2^{k}C_{k}(x,y)-(6x)^{k}\right)B_{n-k}=2n(1-n)(9x^{2}-y)(6x)^{n-2}.\\
\sum_{\underset{n\equiv k[2]}{k=0}}^{n}\binom{n}{k}(1-2^{n-k})(2\sqrt{9x^{2}-y})^{n-k}(n-k-1)\left(2^{k}C_{k}(x,y)+(6x)^{k}\right)B_{n-k}=-2n(1-n)(9x^{2}-y)
(6x)^{n-2}.
\end{gather*}
\end{corollary}

\begin{corollary}
Let $n$ be a positive integer, we have
\begin{gather*}
\sum_{\underset{m\equiv k[2]}{k=0}}^{m}\binom{m}{k}(2\sqrt{9x^{2}-y})^{m-k}\frac{B_{m-k-2}}{m-k-2}\left(2^{k}C_{k}(x,y)-(6x)^{k}\right)=\frac{2^{m+2}(C_{m+2}(x,y)-(6x)^{m+2})}{4(9x^{2}-y)(m+2)(m+1)}-(6x)^{n}.\\
\sum_{\underset{m\equiv k[2]}{k=0}}^{n}\binom{m}{k}(2^{m-k+2}-1)(2\sqrt{9x^{2}-y})^{m-k}\frac{B_{m-k-2}}{m-k-2}\left(2^{k+1}C_{k}(x,y)+2(6x)^{k}\right)=-2(6x)^{m}.
\end{gather*}
\end{corollary}

\begin{corollary}\label{4}
Let $n$ be a positive integer, we have
\begin{equation*}
\sum_{\underset{n\equiv k[2]}{k=0}}^{n}\binom{n}{k}\sqrt{9x^{2}-y}^{n-k}\left(2^{k}C_{k}(x,y)+(6x)^{k}\right)E_{n-k}=(6x+\sqrt{9x^{2}-y})^{n}+(6x-\sqrt{9x^{2}-y})^{n}.
\end{equation*}
\end{corollary}
\begin{itemize}
\item If we put $t=1$ in the corollaries \ref{3}$-$\ref{4}, We can establish the relation  between the Genocchi, Bernoulli, and Euler numbers, as well as the balancing and Lucas-balancing polynomials.
\item If we set $y=t=1$ in the corollaries \ref{3}$-$\ref{4}, we derive the combinatorial identities involving Genocchi, Bernoulli, and Euler numbers with balancing and Lucas-balancing numbers.
\end{itemize}
\end{enumerate}
\section{ \textbf{New convolution of bivariate polynomials and known polynomials }}
In this section, we provide convolutions of the  symmetric function with Genocchi, Bernoulli, and Euler polynomials and we derive  some special cases.
\begin{theorem}\label{gh}
Given an alphabet $A=\{x_{1},x_{2}\}$, the following formulas hold for any positive integer $n$
\begin{gather}
\sum_{k=0}^{n}\binom{n}{k}(x_{1}-x_{2})^{n-k}\left(2^{k}\phi_{k}(x_{1}+x_{2})+2(x_{1}+x_{2})^{k}\right)G_{n-k}(x)
=P_{1}.\label{099}
\end{gather}
With:
$$P_{1}=2n(x_{1}-x_{2})\left((x_{1}+x_{2}+x(x_{1}-x_{2}))^{n-1}+(2x_{2}+x(x_{1}-x_{2}))^{n-1}\right).$$
\end{theorem}
\begin{proof}
We use the change of variable $z=(x_{1}-x_{2})z$ in 
\eqref{105}, we obtain
\begin{equation}
\sum_{n=0}^{\infty}(x_{1}-x_{2})^{n}G_{n}(x)\frac{z^{n}}{n!}=\frac{2z(x_{1}-x_{2})\exp(x(x_{1}-x_{2})z)}{\exp((x_{1}-x_{2})z)+1} \label{dc}
\end{equation}
Take the square of the exponential generating function associated with the symmetric function $\phi_{k}(x_{1}+x_{2})$ and multiply it by relation \eqref{dc}. You will arrive at
\begin{align*}
&\left(\sum_{k=0}^{\infty}\phi_{k}(x_{1}+x_{2})\frac{z^{k}}{k!}\right)^{2}\left(\sum_{n=0}^{\infty}(x_{1}-x_{2})^{n}G_{n}(x)\frac{z^{n}}{n!}\right)\\
&=\left(\sum_{k=0}^{\infty}\left(2^{k}\phi_{k}(x_{1}+x_{2})+2(x_{1}+x_{2})^{k}\right) \frac{z^{k}}{k!}\right)\left(\sum_{n=0}^{\infty}(x_{1}-x_{2})^{n}G_{n}(x)\frac{z^{n}}{n!}\right)\\
&=\sum_{n=0}^{\infty}\sum_{k=0}^{n}\binom{n}{k}(x_{1}-x_{2})^{n-k}\left(2^{k}\phi_{k}(x_{1}+x_{2})+2(x_{1}+x_{2})^{k}\right)G_{n-k}(x)\frac{z^{n}}{n!},
\end{align*}
On the other side, we have
\begin{align*}
&\exp(2x_{2}z)(\exp((x_{1}-x_{2})z)+1)^{2}\frac{2(x_{1}-x_{2})z\exp(x(x_{1}-x_{2})z)}{\exp((x_{1}-x_{2})z)+1}\\
&=2(x_{1}-x_{2})z(\exp((x_{1}+x_{2}+x(x_{1}-x_{2}))z)+\exp((2x_{2}+x(x_{1}-x_{2}))z))\\
&=\sum_{n=0}^{\infty}2n(x_{1}-x_{2})\left((x_{1}+x_{2}+x(x_{1}-x_{2}))^{n-1}+(2x_{2}+x(x_{1}-x_{2}))^{n-1}\right)\frac{z^{n}}{n!}.
\end{align*}

Comparing of the coefficients of $\frac{z^{n}}{n!}$, we obtain the desired result.
\end{proof}
\begin{theorem}
Given an alphabet $A=\{x_{1},x_{2}\}$, the following formulas hold for any positive integer $n$
\begin{gather}
\sum_{k=0}^{n}\binom{n}{k}%
(x_{1}-x_{2})^{n-k-1}\left(2^{k}\phi_{k}(x_{1}+x_{2})-2(x_{1}+x_{2})^{k}\right)B_{n-k}(x) 
=P_{2}.\label{Z35}
\end{gather}
with:
$$P_{2}=n\left((x_{1}+x_{2}+(x_{1}-x_{2})x)^{n-1}-((x_{1}-x_{2})x+2x_{2})^{n-1}\right).$$
\end{theorem}
\begin{proof}
We use the change of variable $z=(e_{1}-e_{2})z$ in 
\eqref{17}, we obtain
\begin{equation}
\sum_{n=0}^{\infty}(x_{1}-x_{2})^{n}B_{n}(x)\frac{z^{n}}{n!}=\frac{(x_{1}-x_{2})z\exp (x(x_{1}-x_{2})z)}{\exp
((x_{1}-x_{2})z)-1}. \label{mbv} 
\end{equation}
Multiply relation \eqref{mbv} by the square of the exponential generating function \eqref{kl}, we have
\begin{align*}
 &\left(\sum\limits_{n=0}^{\infty }(x_{1}-x_{2})^{n}B_{n}(x)\frac{z^{n}}{n!}\right)\left(\sum\limits_{k=0}^{\infty}S_{k-1}(x_{1}+x_{2})\frac{z^{k}}{k!}\right)^{2}\\
&=\left(\sum\limits_{n=0}^{\infty }(x_{1}-x_{2})^{n}B_{n}(x)\frac{z^{n}}{n!}\right)\left(\sum_{k=0}^{\infty}\frac{1}{(x_{1}-x_{2})^{2}}\left(2^{k}\phi_{k}(x_{1}+x_{2})-2(x_{1}+x_{2})^{k}\right) \frac{z^{k}}{k!}\right)\\
&=\sum\limits_{n=0}^{\infty }\sum_{k=0}^{n}\binom{n}{k}%
(x_{1}-x_{2})^{n-k-2}\left(2^{k}\phi_{k}(x_{1}+x_{2})-2(x_{1}+x_{2})^{k}\right)B_{n-k}(x)\frac{z^{n}}{n!}.
\end{align*}
and
\begin{align*}
&\frac{
(x_{1}-x_{2})z\exp ((x_{1}-x_{2})xz)}{\exp ((x_{1}-x_{2})z)-1}\frac{\exp(2x_{2}z)}{(x_{1}-x_{2})^{2}}((\exp(x_{1}-x_{2})z)-1)^{2}\\
&=\frac{z}{x_{1}-x_{2}}\left(\exp ((x_{1}+x_{2}+(x_{1}-x_{2})x)z)-\exp (((x_{1}-x_{2})x+2x_{2})z)\right)\\
&=\sum\limits_{n=0}^{\infty }\frac{n}{x_{1}-x_{2}}\left((x_{1}+x_{2}+(x_{1}-x_{2})x)^{n-1}-((x_{1}-x_{2})x+2x_{2})^{n-1}\right)
\frac{z^{n}}{n!}.
\end{align*}
Comparing of the coefficients of $\frac{z^{n}}{n!}$, we obtain the desired result.
\end{proof}
\begin{theorem}\label{hg}
Given an alphabet $A=\{x_{1},x_{2}\}$, the following formulas hold for any positive integer $n$
\begin{align}
\sum_{k=0}^{n}\binom{n}{k}(x_{1}-x_{2})^{n-k}(2^{k}\phi_{k}(x_{1}+x_{2})+2(x_{1}+x_{2})^{k})E_{n-k}(x)=P_{3}. \label{adl}
\end{align}
with:
$$P_{3}=2\left((x_{1}+x_{2}+(x_{1}-x_{2})x)^{n}
+(2x_{2}+(x_{1}-x_{2})x)^{n}\right).$$
\end{theorem}
\begin{proof}
We use the change of variable $z=(x_{1}-x_{2})z$ in 
\eqref{18}, we obtain
\begin{equation}
\sum_{n=0}^{\infty}(x_{1}-x_{2})^{n}E_{n}(x)\frac{z^{n}}{n!}=\frac{2\exp (x(x_{1}-x_{2})z)}{\exp
((x_{1}-x_{2})z)+1}. \label{ghhd}
\end{equation}
Take the square of the exponential generating function associated with the symmetric function\eqref{l} and multiply it by relation \eqref{ghhd}. You will arrive at
\begin{align*}
&\left(\sum_{n=0}^{\infty}(x_{1}-x_{2})^{n}E_{n}(x)\frac{z^{n}}{n!}\right)\left(\sum_{k=0}^{\infty}\phi_{k}(x_{1}+x_{2})\frac{z^{k}}{k!}\right)^{2}\\
&\left(\sum_{n=0}^{\infty}(x_{1}-x_{2})^{n}E_{n}(x)\frac{z^{n}}{n!}\right)\left(\sum_{k=0}^{\infty}(2^{k}\phi_{k}(x_{1}+x_{2})+2(x_{1}+x_{2})^{k})\frac{z^{k}}{k!}\right)\\
&=\sum_{n=0}^{\infty}\sum_{k=0}^{n}\binom{n}{k}(x_{1}-x_{2})^{n-k}(2^{k}\phi_{k}(x_{1}+x_{2})+2(x_{1}+x_{2})^{k})E_{n-k}(x)\frac{z^{n}}{n!}.
\end{align*}
then
\begin{align*}
\frac{2\exp
((x_{1}-x_{2})xz)}{\exp ((x_{1}-x_{2})z)+1}(\exp(x_{1}z)+\exp(x_{2}z))^{2}
&=2\left(\exp
((x_{1}+x_{2}+(x_{1}-x_{2})x)z)+\exp
((2x_{2}+(x_{1}-x_{2})x)z)\right)\\
&=\sum\limits_{n=0}^{\infty }2(x_{1}+x_{2}+(x_{1}-x_{2})x)^{n}
\frac{z^{n}}{n!}+\sum\limits_{n=0}^{\infty }2(2x_{2}+(x_{1}-x_{2})x)^{n}
\frac{z^{n}}{n!}\\
&=\sum\limits_{n=0}^{\infty }2\left((x_{1}+x_{2}+(x_{1}-x_{2})x)^{n}+2(2x_{2}+(x_{1}-x_{2})x)^{n}\right)
\frac{z^{n}}{n!}.
\end{align*}
Comparing of the coefficients of $\frac{z^{n}}{n!}$, we obtain the desired result.
\end{proof}
\subsection{\textbf{Some special cases}}

We now consider the Theorems \ref{gh}-\ref{hg} to derive the following two cases.
\begin{enumerate}
\item [\textbf{Case 1.}]
Let $x_{1}=\frac{y+\sqrt{y^{2}+4t}}{2}$ and $x_{2}=\frac{y-\sqrt{y^{2}+4t}}{2}$, then we have the following the following combinatorial
results of Genocchi, Bernoulli and Euler  polynomials with bivariate Fibonacci,  bivariate Lucas polynomials.
\begin{corollary}\label{5}
Let $n$ be a positive integer, we have
\begin{align*}
&\sum_{k=0}^{n}\binom{n}{k}\sqrt{y^{2}+4t}^{n-k}\left(2^{k}L_{k}(y,t)+2y^{k}\right)G_{n-k}(x)\\
&=2n\left((y+x\sqrt{y^{2}+4t})^{n-1}+(y+(x-1)\sqrt{y^{2}+4t})^{n-1}\right).
\end{align*}
\end{corollary}
\begin{corollary}
Let $n$ be a positive integer, we have
\begin{align*}
\sum_{k=0}^{n}\binom{n}{k}%
\sqrt{y^{2}+4t}^{n-k-1}\left(2^{k}L_{k}(y,t)-2y^{k}\right)B_{n-k}(x)
&=n(y+\sqrt{y^{2}+4t}x)^{n-1}\\
&-n(y+(x-1)\sqrt{y^{2}+4t})^{n-1}.
\end{align*}
\end{corollary}
\begin{corollary}\label{6}
Let $n$ be a positive integer, we have
\begin{align*}
\sum_{k=0}^{n}\binom{n}{k}\sqrt{y^{2}+4t}^{n-k}(2^{k}L_{k}(y,t)+2y^{k})E_{n-k}(x)&=2(y+\sqrt{y^{2}+4t}x)^{n}\\
&+2(y+(x-1)\sqrt{y^{2}+4t})^{n}.\notag
\end{align*}
\end{corollary}
\begin{itemize}
\item If we put $t=1$ in the corollaries \ref{5}$-$\ref{6}, we can   obtain the connections between Genocchi, Bernoulli and Euler  polynomials and
 Fibonacci and Lucas polynomials.
\item If we set $y=t=1$ in the corollaries \ref{5}$-$\ref{6}, We derive connections between Genocchi, Bernoulli and Euler polynomials and
 Fibonacci and Lucas numbers.
\end{itemize}
\item [\textbf{Case 2.}]
Let $x_{1}=3y+\sqrt{9y^{2}-t}$ and $x_{2}=3y-\sqrt{9y^{2}-t}$, then we have the following the following 
results of Genocchi, Bernoulli and Euler  polynomials with bivariate balancing,  bivariate Lucas-balancing polynomials.
\begin{corollary}\label{7}
Let $n$ be a positive integer, we have
\begin{align*}
&\sum_{k=0}^{n}\binom{n}{k}(2\sqrt{9y^{2}-t})^{n-k}\left(2^{k+1}C_{k}(y,t)+2(6y)^{k}\right)G_{n-k}(x)\\
&=n2^{n+1}\sqrt{9y^{2}-t}\left((3y+x\sqrt{9y^{2}-t})^{n-1}+(3y+(x-1)\sqrt{9y^{2}-t})^{n-1}\right).
\end{align*}
\end{corollary}

\begin{corollary}\label{cr10}
Let $n$ be an positive integer, we have
\begin{align*}
\sum_{k=0}^{n}\binom{n}{k}%
(2\sqrt{9y^{2}-t})^{n-k-1}\left(2^{k+1}C_{k}(y,t)-(6y)^{k}\right)B_{n-k}(x)
&=n(6y+2x\sqrt{9y^{2}-t})^{n-1}\\
&-n(6y+(2x-2)\sqrt{9y^{2}-t})^{n-1}.\notag
\end{align*}
\end{corollary}

\begin{corollary}\label{8}
Let $n$ be a positive integer, we have
\begin{align*}
\sum_{k=0}^{n}\binom{n}{k}(2\sqrt{9y^{2}-t})^{n-k}(2^{k+1}C_{k}(y,t)+2(6y)^{k})E_{n-k}(x)&=2(6y+2\sqrt{9y^{2}-t}x)^{n}\\
&+2(6y+(x-2)\sqrt{9y^{2}-t})^{n}.\notag
\end{align*}
\end{corollary}
\begin{itemize}
\item If we put $t=1$ in the corollaries \ref{7}$-$\ref{8}, we    obtain the connections between the Genocchi, Bernoulli and Euler  polynomials and
 balancing and Lucas- balancing polynomials.
\item If we set $y=t=1$ in the corollaries \ref{7}$-$\ref{8}, We derive connections between  the Genocchi, Bernoulli and Euler  polynomials and
balancing and Lucas- balancing numbers.
\end{itemize}
\end{enumerate}
\section{Conclusion}

In this paper, we introduced several new 
convolution sums formulas of symmetric
functions,  Bernoulli, Euler and Genocchi numbers and polynomials.
From these formulas we deduce some special cases of bivariate polynomials  such as bivariate Fibonacci,  bivariate Lucas,  bivariate balancing and  bivariate Lucas-balancing polynomials.\\

\noindent\textbf{Acknowledgements}\\
The authors wish to express their gratitude to the referees for their corrections and numerous valuable suggestions.

\end{document}